\RequirePackage{fix-cm}
\documentclass[a4paper,12pt]{article}      
\usepackage{graphicx}
\usepackage{amsfonts}
\usepackage{amsmath}
\usepackage{amsthm}
\usepackage{epstopdf}
\usepackage{cite}
\newtheorem{theorem}{Theorem}
\DeclareMathOperator{\Var}{Var}
\DeclareMathOperator{\osc}{osc}
\begin{document}

\title{Discrete operators associated with linking operators}  
\author{Margareta Heilmann, Fadel Nasaireh, Ioan Ra\c{s}a }        

\maketitle

\begin{abstract}
We associate to an integral operator a discrete one which is 
conceptually simpler, and study the relations between them.
\end{abstract}
{\bf Keywords:} {Linking operators, Baskakov-Durrmeyer type operators, Kantorovich modifications of operators}
\\
{\bf MSC 2010:} {41A36 \and 41A35 \and 41A28}

\section{Introduction and notation}
\label{sec1}
It is well-known that integral operators are very useful in 
Approximation Theory. However, sometimes they are expressed in terms of 
complicated integrals with respect to complicated measures. In this 
paper we associate to an integral operator a discrete one which is 
conceptually simpler, and study the relations between them. Some results 
in this direction can be found also in \cite{Rasa2011}.

After introducing the necessary definitions we present some general 
results. Then we recall the definitions of Baskakov type operators, 
genuine Baskakov-Durrmeyer type operators, and their Kantorovich 
modifications. We construct the discrete operators associated with these 
integral operators and apply our general results in this context.

\section{Preliminaries}
\label{sec2}
Let $I \subset \mathbb{R}$ be an interval and $H$ a subspace of $C(I)$ containing $e_0$, $e_1$ and $e_2$. Let $L: H \longrightarrow C(I)$ be a positive linear operator such that $Le_0=e_0$. The second moment of $L$ is defined by
$$
	M_2 L(x) = L(e_1-xe_0)^2 (x),  \, x \in I.
$$
For a fixed $x \in I$ consider the functional $H \ni f \longrightarrow Lf(x)$ and define
$\Var_x L := Le_2 (x) -(Le_1 (x))^2$; then, roughly speaking, $\Var_x L$ shows how far is the functional from being a pointwise evaluation.

It is easy to verify that
\begin{equation}
\label{eq2-1}
	M_2L(x)-\Var_xL = (Le_1(x)-x)^2 .
\end{equation}
Now let $L$ be of the form
\begin{equation}
\label{eq2-2}
	Lf: = \sum_{j=0}^\infty A_j (f) p_j, \, f \in H,
\end{equation}
where $A_j : H \longrightarrow \mathbb{R}$ are positive linear functionals,
\begin{equation}
\label{eq-2a}
A_j(e_0) = 1 \mbox{ and } p_j \in C(I), \,  p_j \geq 0, \,  \sum_{j=0}^\infty p_j =e_0 .
\end{equation}
Let 
\begin{equation}
\label{eq2-3}
	b_j :=A_j(e_1), \, \Var A_j :=A_j(e_2)-b_j^2, \, j \geq 0.     
\end{equation}
Then, generally speaking, $\Var A_j$ shows how far is $A_j$ from the point evaluation at $b_j$.

The discrete operator associated with $L$ is defined by 
\begin{equation}
\label{eq2-4}
		D: H \longrightarrow C(I), \, Df :=\sum_{j=0}^\infty f(b_j) p_j.
\end{equation}
The point evaluation functional at $b_j$ is conceptually simpler than $A_j$;
from this point of view, $D$ is simpler than $L$.
We shall investigate the relations between $L$ and $D$.

It is easy to verify that
\begin{equation}
\label{eq2-5}
	M_2 D(x) = \sum_{j=0}^\infty (b_j-x)^2 p_j(x), \, x \in I.  
\end{equation}
Moreover, according to (\ref{eq2-1}) and (\ref{eq-2a})
\begin{eqnarray*}
	M_2L(x)- \Var_x L
		& = &
		(Le_1(x)-x)^2
		\\
		& = &
		\left ( \sum_{j=0}^\infty (b_j-x) p_j(x) \right )^2
		\\
		& = &
		\left ( \sum_{j=0}^\infty (b_j-x) \sqrt{p_j(x)}  \sqrt{p_j(x)} \right )^2
		\\
		& \leq &
		\sum_{j=0}^\infty (b_j-x)^2 p_j(x) .
\end{eqnarray*}
Combined with (\ref{eq2-5}), this shows that 
\begin{equation}
\label{eq2-6}
	0 \leq M_2L(x) -\Var_x L \leq M_2 D(x), \, x \in I.
\end{equation}

Now define
\begin{equation}
\label{eq2-7}
	E(L)(x) := \sum_{j=0}^\infty (\Var{A_j} ) p_j (x), \, x \in I.
\end{equation}
We have
\begin{eqnarray*}
	M_2L(x) 
	& = &
	L(e_1-xe_0)^2 (x)
	\\
	& = &
	\sum_{j=0}^\infty \left ( A_j(e_2)-2xb_j+x^2 \right )p_j(x)
	\\
	& = &
	\sum_{j=0}^\infty \left ( A_j(e_2)-b_j^2 \right )p_j(x) + \sum_{j=0}^\infty (b_j-x)^2p_j(x) .
\end{eqnarray*}
With (\ref{eq2-3}) and (\ref{eq2-5}) this leads to
\begin{equation}
\label{eq2-8}
	M_2L(x) =E(L) (x)+M_2D(x), \, x \in I.
\end{equation}
Combined with (\ref{eq2-6}), this yields
\begin{equation}
\label{eq2-9}
	E(L)(x) \leq \Var_x{L}, \, x \in I.
\end{equation}
Finally, let $f \in H \cap C^2(I)$ and suppose that $\|f''\|_\infty < \infty$. Then by Taylor's formula,
$$
	|f(t)-f(b_j)-(t-b_j)f'(b_j) | \leq \frac{1}{2} (t-b_j)^2 \|f''\|_\infty , \, t \in I.
$$
This entails
\begin{equation}
\label{eq2-10}
	|A_j(f) - f(b_j) | \leq \frac{1}{2} (\Var{A_j}) \|f''\|_\infty .
\end{equation}
Moreover, according to (\ref{eq2-2}) and (\ref{eq2-4}),
$$
	|Lf-Df| \leq \sum_{j=0}^\infty |A_j(f)-f(b_j) | p_j,
$$
and so 
$$
	|Lf-Df| \leq \frac{1}{2}  \|f''\|_\infty \sum_{j=0}^\infty (\Var{A_j}) p_j.
$$
We conclude that
\begin{equation}
\label{eq2-11}
	|Lf(x)-Df(x) | \leq \frac{1}{2}  \|f''\|_\infty E(L)(x), \, x \in I.
\end{equation}
Using (\ref{eq2-11}) we see that $E(L)(x)$ shows how far is $L$ from $D$.

\section{Linking operators and discrete operators}
In \cite{Paltanea2007, Paltanea2008}  P\u {a}lt\u {a}nea introduced operators depending on a parameter $\rho \in \mathbb{R}^+ $, which constitute a non-trivial link between the Bernstein and Sz\'{a}sz-Mirakjan operators, respectively, and their genuine Durrmeyer modifications. In \cite{HeilmannRasa2015} this definition was extended to 
 a non-trivial link between Baskakov type operators and genuine Baskakov-Durrmeyer type operators and their $k$-th order Kantorovich modification.
 For $k=1$ this means a link between the Kantorovich modification of Baskakov type and Baskakov-Durrmeyer type operators.

In what follows for $c \in \mathbb{R} $ we use the notations
$$
	a^{c,\overline{j}} := \prod_{l=0}^{j-1}  (a+cl) , \; a^{c,\underline{j}} := \prod_{l=0}^{j-1}  (a-cl) , \; j \in \mathbb{N}; \quad
	 a^{c,\overline{0}}= a^{c,\underline{0}} :=1 .
$$
 
In the following definitions of the operators we omit the parameter $c$ in the notations in order to reduce the necessary sub and superscripts.

Let $c \in \mathbb{R}$, $n \in \mathbb{R}$, $n > c$ for $c\geq0$ and $-n/c \in \mathbb{N}$ for $c<0$. Furthermore let  $\rho \in \mathbb{R}^+$, $j \in \mathbb{N}_0$, $x \in I_c$ with $I_c = [0,\infty)$ for $c\geq0$ and $I_c=[0,-1/c]$ for $c < 0$. Then the basis functions are given by
$$
	p_{n,j}(x) = \left \{ \begin{array}{cl}
	\frac{n^j}{j!} x^j  e^{-nx} &, \, c = 0 ,\\
	\frac{n^{c,\overline{j}}}{j!} x^j (1+cx)^{-\left ( \frac{n}{c}+j\right)} &, \, c > 0 .
 	\end{array} \right .
$$
In the following definitions we assume that the function $f$ is given in such a way that the corresponding integrals and series are convergent.
The operators of Baskakov-type are defined by
\begin{equation}
\label{eq0.1a}
	(B_{n,\infty} f)(x) = \sum_{j=0}^{\infty} p_{n,j}(x) f \left ( \frac{j}{n} \right ) ,
\end{equation}
and 
the genuine Baskakov-Durrmeyer type operators are denoted  by 
\begin{eqnarray}
\label{eq0.2}
	 (B_{n,1} f)(x)
	& = & 
	f(0) p_{n,0}(x)  + f\left( -\frac{1}{c} \right ) p_{n,-\frac{n}{c}} (x)
\\
\nonumber
	& & 
	+ \sum_{j=1}^{-\frac{n}{c}-1} p_{n,j}(x) (n+c) \int_0^{-\frac{1}{c}} p_{n+2c,j-1} (t) f (t) dt 
\end{eqnarray}
for $c<0$ and by
\begin{eqnarray*}
	 (B_{n,1} f)(x)
	& = & 
	f(0) p_{n,0}(x)  +
	\sum_{j=1}^{\infty} p_{n,j}(x) (n+c) \int_0^{\infty} p_{n+2c,j-1} (t) f (t) dt
\end{eqnarray*}
for $c \geq 0$.
\\
Depending on a parameter  $\rho \in \mathbb{R}^+$ the linking operators are given by
\begin{equation}
\label{eq0.3}
	 (B_{n,\rho} f)(x) = \sum_{j=0}^{\infty} F_{n,j}^\rho (f) p_{n,j} (x)
\end{equation}
where 
$$
	F_{n,j}^\rho (f) = \left \{
	\begin{array}{cll}
	f(0) & , & j=0, \, c \in \mathbb{R},
	\\
	\displaystyle f \left (-\frac{1}{c} \right ) & , & j=-\frac{n}{c} , \, c <0,
	\\
	\displaystyle \int_{I_c} \mu_{n,j}^{\rho} (t) f(t) dt & , & \mbox{otherwise},
	\end{array} \right .
$$
with
$$
	\mu_{n,j}^{\rho} (t) = \left \{ \begin{array}{cl}
	\displaystyle \frac{(-c)^{j\rho}}{B \left (j\rho,- \left (\frac{n}{c}+j \right )\rho \right )} t^{j\rho -1} (1+ct)^{-\left ( \frac{n}{c}+j\right)\rho -1} &, \, c < 0 , \\
	\displaystyle \frac{(n\rho)^{j\rho}}{\Gamma (j \rho )} t^{j\rho-1}  e^{-n\rho t} &, \, c = 0  ,\\
	\displaystyle \frac{c^{j\rho}}{B \left (j\rho,\frac{n}{c}\rho+1 \right )} t^{j\rho -1} (1+ct)^{-\left ( \frac{n}{c}+j\right)\rho -1} &, \, c > 0 .
	\end{array} \right .
$$
By $B(x,y) = \int_0^1 t^{x-1} (1-t)^{y-1} dt$, $x,y > 0$, we denote Euler's Beta function.

Note that in case $c<0$ the sums in (\ref{eq0.1a}) and (\ref{eq0.3}) are finite, as $p_{n,j}(x) = 0$ for $j>-n/c$.
The $k$-th order Kantorovich modification of the operators $B_{n,\rho}$ are defined by
$$
	B_{n,\rho}^{(k)}:=D^k \circ B_{n,\rho}\circ I_k
$$
where
$D^k$ denotes the $k$-th order ordinary differential operator and
$$
	I_k f = f, \mbox{ if } k=0, \mbox{ and } (I_k f)(x) = \int_0^x \frac{(x-t)^{k-1}}{(k-1)!}  f(t) dt,
	\mbox{ if } k \in \mathbb{N}.
$$
For $k=0$ we omit the superscript $(k)$ as indicated by the definition above.

We recall some results concerning $\lim_{\rho \to \infty } B_{n,\rho}^{(k)}$.

In \cite[Theorem 2.3]{GonskaPaltanea2010-1} Gonska and P\u {a}lt\u {a}nea proved for $c=-1$ the convergence of the operators $B_{n,\rho}$ to the classical Bernstein operator $B_{n,\infty}$, i.e., they proved that for every
$f \in C[0,1]$
$$
	\lim_{\rho \to \infty}{B_{n,\rho}} f = B_{n,\infty} f \mbox{ uniformly on }  [0,1].
$$
From \cite{HeilmannRasa2014} (see the consideration of the special case $\rho \to \infty$ after Remark 2 there) we know that for each polynomial $q$
$$
	\lim_{\rho \to \infty}B_{n,\rho}^{(k)} q = B_{n,\infty}^{(k)} q \mbox{ uniformly on }  [0,1].
$$
Let $ \varepsilon > 0$ be arbitrary. As the space of polynomials $ \mathcal{P}$ is dense in 
$ L_p[0,1]$, $\| \cdot \|_p $, $1 \leq p < \infty$ and $C[0,1]$, $\| \cdot \|_\infty$, $p=\infty$, we can choose a polynomial $q$, such that $\|f-q\|_p < \varepsilon$. Then
\begin{eqnarray*}
	\| (B_{n,\rho}^{(k)} - B_{n,\infty}^{(k)}) f \|_p
	& \leq &
	\| B_{n,\rho}^{(k)} (f - q) \|_p
	+\| B_{n,\infty}^{(k)} (f - q)  \|_p
	+\| (B_{n,\rho}^{(k)}  - B_{n,\infty}^{(k)} ) q\|_p .
\end{eqnarray*}
As the operators $B_{n,\rho}^{(k)} $ and $B_{n,\infty}^{(k)}$ are bounded (see \cite[Corollary 1]{HeilmannRasa2014} and \cite[(3)]{GonskaHeilmannRasa2011} for the images of $e_0=1$) we immediately get
\begin{equation}
\label{eq00}
	\lim_{\rho \to \infty}
	\|(B_{n,\rho}^{(k)} - B_{n,\infty}^{(k)}) f \|_p = 0
\end{equation}
for each $ f \in  L_p[0,1]$, $\| \cdot \|_p $, $1 \leq p < \infty$ and $C[0,1]$, $\| \cdot \|_\infty$, $p=\infty$.

We will use the following representations for the linking operators (see \cite[Theorem 2]{HeilmannRasa2017} and \cite[Theorem 4]{HeilmannRasa2018} in case $\rho \in \mathbb{N}$.

$c=-1$: Let $n,k \in \mathbb{N}$, $n-k \geq 1$, $\rho \in \mathbb{N}$ and $f \in L_1[0,1]$. Then we have the representation
\begin{eqnarray}
\label{eq15a}
	{B_{n,\rho}^{(k)} (f;x)}
	& = &
	\frac{n! (n\rho -1)!}{(n-k)! (n\rho+k-2)!} \sum_{j=0}^{n-k} p_{n-k,j} (x)
\\
\nonumber
	& &
	\times 
	\int_0^1
	\sum_{i_1=0}^{\rho-1} \dots \sum_{i_k=0}^{\rho-1}
	p_{n\rho+k-2,j\rho +i_1 + \dots +i_k+k-1} (t)   f(t) dt .
\end{eqnarray}

$c \geq 0$: Let $n,k \in \mathbb{N}$, $n-k \geq 1$, $\rho \in \mathbb{N}$ and $f \in W_n^\rho$. Then we have the representation
\begin{eqnarray}
\label{eq15b}
	{B_{n,\rho}^{(k)} (f;x)}
	& = &
	\frac{n^{c,\overline{k}}}{(n\rho)^{c,\underline{k-1}}} \sum_{j=0}^{\infty} p_{n+kc,j} (x)
\\
\nonumber
	& &
	\times 
	\int_0^\infty
	\sum_{i_1=0}^{\rho-1} \dots \sum_{i_k=0}^{\rho-1}
	p_{n\rho-c(k-2),j\rho +i_1 + \dots +i_k+k-1} (t)   f(t) dt .
\end{eqnarray}
For the images of monomials explicit representations are known from \cite[Theorem 2]{HeilmannRasa2014} and \cite[Theorem 2]{HeilmannRasa2015}. 
We will need (see \cite[Corollary 1]{HeilmannRasa2014} and \cite[Corollary 2]{HeilmannRasa2015} the images of the first monomials.
\begin{eqnarray}
\label{moexeq4}
	(B_{n,\rho}^{(k)} e_0 ) (x)
	& = &
	\frac{\rho^k}{(n\rho)^{c,\underline{k}}} \cdot n^{c,\overline{k}},
	\\
	\nonumber
	(B_{n,\rho}^{(k)}e_1 ) (x)
	& = &
	\frac{\rho^{k+1}}{(n\rho)^{c,\underline{k+1}}} \cdot n^{c,\overline{k}} \left [ \frac{1}{2}k\left ( 1 + \frac{1}{\rho}\right ) + (n+ck) x \right ],
		\\
		\nonumber
	(B_{n,\rho}^{(k)}e_2) (x)
	& = &
	\frac{\rho^{k+2}}{(n\rho)^{c,\underline{k+2}}} \cdot n^{c,\overline{k}} 
	\left [ \frac{1}{2}k \left ( \frac{3k+1}{6} + \frac{k+1}{ \rho} + \frac{3k+5}{6 \rho^2} \right ) \right .
	\\
	\nonumber
	& &
	\quad \left .+ (n+ck) \left ( (k+1) \left ( 1 + \frac{1}{\rho}\right ) x  + (n+c(k+1)) x^2 \right ) \right ].
\end{eqnarray}
%

%


We consider the operators
\begin{equation}
\label{eq3-1}
	V_{n,\rho}^{(k)} := \frac{(n\rho)^{c,\underline{k}}}{\rho^k n^{c,\overline{k}}} B_{n,\rho}^{(k)}
\end{equation}
for which we have $V_{n,\rho}^{(k)} e_0 = e_0$. 

They are of the form (\ref{eq2-2}); more precisely,
\begin{equation}
\label{eq3-2}
	V_{n,\rho}^{(k)}f: = \sum_{j=0}^\infty A_{n,\rho,j}^{(k)} (f) p_{n+kc,j}, 
\end{equation}
where (see (\ref{eq15a}), (\ref{eq15b}) and (\ref{eq3-1}))
\begin{eqnarray}
\label{eq3-3}
	\lefteqn{A_{n,\rho,j}^{(k)} (f) }
	\\
	\nonumber
	&:=&
	\frac{n\rho-(k-1)c}{\rho^k} \sum_{i_1=0}^{\rho-1} \dots \sum_{i_k=0}^{\rho-1}
	\int_0^\infty p_{n\rho-(k-2)c,j\rho+i_1+ \dots + i_k+k-1} (t) f(t) dt .
\end{eqnarray}
If $n\rho > kc$, we can consider the barycenter of $A_{n,\rho,j}^{(k)}$.
As $t p_{m+c,l-1}(t) = \frac{l}{m} p_{m,l}(t)$ and $ \int_0^\infty p_{m,l}(t) dt = \frac{1}{m+1}$
we can calculate
\begin{eqnarray}
\label{eq3-4}
	\lefteqn{b_{n,\rho,j}^{(k)}:= A_{n,\rho,j}^{(k)} (e_1) }
\\
	\nonumber
	& = & \frac{n\rho-kc+c}{\rho^k} 
	\sum_{i_1=0}^{\rho-1} \dots \sum_{i_k=0}^{\rho-1} \int_0^\infty
	p_{n\rho-c(k-2),j\rho +i_1 + \dots +i_k+k-1} (t)   t dt 
\\
\nonumber
& = &  \frac{n\rho-kc+c}{\rho^k} 
	\sum_{i_1=0}^{\rho-1} \dots \sum_{i_k=0}^{\rho-1} \frac{j\rho+i_1+ \dots i_k +k}{n\rho-ck+c}
	\cdot \frac{1}{n\rho-ck}
\\
\nonumber
	& = & \frac{1}{\rho^k(n\rho-ck)} 
	\left \{ \sum_{i_1=0}^{\rho-1} \dots \sum_{i_k=0}^{\rho-1} (j\rho+k)
		+ \sum_{i_1=0}^{\rho-1} \dots \sum_{i_k=0}^{\rho-1} (i_1+ \dots i_k ) \right \}
\\
\nonumber
	& = & \frac{1}{\rho^k(n\rho-ck)} \left \{ (j\rho+k)\rho^k +\rho^{k-1} \cdot \frac{(\rho-1)\rho}{2}\cdot k \right \}
\\
\nonumber
	& = & \frac{(2j+k)\rho+k}{2(n\rho-kc)} .
\end{eqnarray}
The discrete operators (\ref{eq2-4}) associated with $V_{n,\rho}^{(k)}$ are
\begin{equation}
\label{eq3-5}
	D_{n,\rho}^{(k)} f := \sum_{j=0}^\infty f(b_{n,\rho,j}^{(k)}) p_{n+kc,j} .
\end{equation}
Under the form
$$
	D_{n,\rho}^{(k)} f  = \sum_{j=0}^\infty f \left ( \frac{j+\frac{(\rho+1)k}{2\rho}}{(n+kc)-\frac{\rho+1}{\rho}kc} \right ) p_{n+kc,j},
$$
we see that $D_{n,\rho}^{(k)}$ is a Stancu-type modification of the operator (\ref{eq0.1a})
\begin{equation}
\label{eq3-6}
B_{n+kc,\infty} f = \sum_{j=0}^\infty f \left ( \frac{j}{n+kc}  \right ) p_{n+kc,j}.
\end{equation}
A direct calculation similar to (\ref{eq3-4}) shows that
\begin{eqnarray}
\label{eq3-7}
	\lefteqn{A_{n,\rho,j}^{(k)} (e_2) }
	\\
	\nonumber
	&:=&
	\frac{12(j\rho+k)(j\rho+k+1)+4k[(3j+1)\rho+3k+1](\rho-1)+3k(k-1)(\rho-1)^2}{12(n\rho-kc)(n\rho-kc-c)} .
\end{eqnarray}
Now (\ref{eq2-3}), (\ref{eq3-4}) and (\ref{eq3-7}) imply
\begin{eqnarray}
\label{eq3-8}
	\lefteqn{\Var{A_{n,\rho,j}^{(k)}}  }
	\\
	\nonumber
	&:=&
	\frac{kn\rho^3+6[k(2jc+n)+2j(jc+n)]\rho^2+5kn\rho+2k^2c(\rho^2-1)}{12(n\rho-kc)^2(n\rho-kc-c)} .
\end{eqnarray}
With the notation $E_{n,\rho}^{(k)}:= E \left ( V_{n,\rho}^{(k)} \right )$ and using (\ref{eq2-7}) and (\ref{eq3-8}), we get
\begin{eqnarray}
\label{eq3-9}
	\lefteqn{E_{n,\rho}^{(k)} (x)}
	\\
	\nonumber
	&:=&
	\frac{kn\rho^3+[12x(1+cx)(kc+n)(kc+n+c)+6kn]\rho^2+5kn\rho+2k^2c(\rho^2-1)}{12(n\rho-kc)^2(n\rho-kc-c)} .
\end{eqnarray}
Here are some particular values:	
		\begin{align}
		\label{eq3-a}
			\Var{A_{n,\infty ,j}^{(k)}} &= \frac{k}{12n^2} 
			, &\! \! \!
			\Var{A_{n,1 ,j}^{(k)}} &= \frac{(j+k)(n+jc)}{(n-kc)^2(n-kc-c)}
			\\	
		\label{eq3-b}
			\mbox{and } E_{n,\infty}^{(k)}(x) &=\frac{k}{12n^2}
			, & \! \! \!
			E_{n,1}^{(k)} (x) &=\frac{x(1+cx)(n+kc)(n+(k+1)c)+kn}{(n-kc)^2(n-(k+1)c)}
		\end{align}
For $c=0$, 
		\begin{align}
		\label{eq-c}
			\Var{A_{n,\rho ,j}^{(k)}} &= \frac{k\rho^2+6(k+2j)\rho+5k}{12n^2\rho^2}
			,& E _{n,\rho}^{(k)}  =  \frac{12n\rho  x+k(\rho+1)(\rho+5)}{12n^2 \rho^2} .
		\end{align}
The images of the first monomials under $V_{n,\rho}^{(k)} $ can be deduced from (\ref{moexeq4}); using them we find
\begin{eqnarray}
	&& Var_x{V_{n,\rho}^{(k)}} := V_{n,\rho}^{(k)} e_2 (x) -\left ( V_{n,\rho}^{(k)} e_1 (x) \right )^2
	\\
	& = &
	\nonumber
	\frac{12nx(1+cx)(n+kc)\rho^2(\rho+1)+2k(3n+kc)\rho^2+kn\rho(\rho^2+5)-2k^2c}{12(n\rho-kc)^2(n\rho-kc-c)} .
\end{eqnarray}
As explained in Section \ref{sec2}, $\Var_x{V_{n,\rho}^{(k)}}$ shows how far is the functional 
$f \mapsto  V_{n,\rho}^{(k)} f(x)$ from the point evaluation at its barycenter 
$\frac{2(n+kc)\rho x +k(\rho +1)}{2(n\rho -kc)}$.

This assertion, and other similar ones, are made more precise in the following Theorem.
\begin{theorem}
\label{thm3-1}
Let $\|f''\|_\infty < \infty$. Then
\begin{eqnarray}
\label{eq3-10}
	\left | A_{n,\rho ,j}^{(k)} (f) - f \left ( \frac{(2j+k)\rho +k}{2(n\rho-kc)} \right ) \right |
	&\leq &
	\frac{1}{2} \|f''\|_\infty \Var{A_{n,\rho ,j}^{(k)}} ,
	\\
\label{eq3-11}
	\left | V_{n,\rho}^{(k)} f(x) - D_{n,\rho}^{(k)} f(x) \right |
	&\leq &
	\frac{1}{2} E_{n,\rho}^{(k)} (x) \|f''\|_\infty
	\\
\label{eq3-12}
	\left | V_{n,\rho}^{(k)} f(x) - f \left ( \frac{2(n+kc)\rho x+k(\rho+1)}{2(n\rho-kc)} \right ) \right |
	& \leq &
	\frac{1}{2} \|f''\|_\infty \Var_x{V_{n,\rho}^{(k)}} .
\end{eqnarray}
\end{theorem}
\begin{proof}
It suffices to apply (\ref{eq2-10}) and  (\ref{eq2-11}) to the functionals 
$A_{n,\rho ,j}^{(k)}$ and $f \mapsto V_{n,\rho}^{(k)} f(x)$, respectively to the operators $V_{n,\rho}^{(k)}$.
\end{proof}
The case $\rho \to \infty$ deserves a special attention.
\begin{theorem}
\label{thm3-2}
\begin{itemize}
\item[(i)] If $\|f''\|_\infty < \infty$, then
		\begin{equation}
		\label{eq3-14}
		\left \| V_{n,\infty}^{(k)} f - D_{n,\infty}^{(k)} f \right \|_\infty 
		\leq
		\frac{k}{24n^2}  \|f''\|_\infty .
		\end{equation}
\item[(ii)] For an arbitrary $f$ in the domain of $V_{n,\infty}^{(k)}$,
		\begin{equation}
		\label{eq3-15}
		\left \| V_{n,\infty}^{(k)} f - D_{n,\infty}^{(k)} f \right \|_\infty 
		\leq
		\omega \left ( f; \frac{k}{2n} \right ) .
		\end{equation}		
\end{itemize}
\end{theorem}
\begin{proof}
(\ref{eq3-14}) follows from (\ref{eq3-12}) and (\ref{eq3-b}).
In order to prove (\ref{eq3-15}), let us remark that 
\begin{equation}
\label{eq3-16}
	D_{n,\infty}^{(k)} f(x) = \sum_{j=0}^\infty f \left ( \frac{2j+k}{2n} \right ) p_{n+kc,j} (x) ,
\end{equation}
and
\begin{eqnarray*}
	V_{n,\infty}^{(k)} f(x) 
	& = &
	\frac{n^k}{n^{c,\overline{k}}} B_{n,\infty}^{(k)} f(x)
	\\
	& = &
	n^k \sum_{j=0}^\infty p_{n+kc,j} (x) \Delta^k_{\frac{1}{n}} (I_kf)\left ( \frac{j}{n} \right )
	\\
	& = &
	k! \sum_{j=0}^\infty \left [ \frac{j}{n},  \frac{j+1}{n}, \dots , \frac{j+k}{n}; I_kf \right ]
	p_{n+kc,j} (x).
\end{eqnarray*}
According to the mean value theorem for divided differences there exists $x_{n,k,j} \in 
\left [\frac{j}{n},  \frac{j+k}{n} \right ]$ such that
$$
	k!  \left [ \frac{j}{n},  \frac{j+1}{n}, \dots , \frac{j+k}{n}; I_kf \right ] = f \left (x_{n,k,j} \right ).
$$
Consequently,
\begin{equation}
\label{eq3-17}
	V_{n,\infty}^{(k)} f(x)  = \sum_{j=0}^\infty f \left (x_{n,k,j} \right ) p_{n+kc,j} (x) .
\end{equation}
Since $x_{n,k,j} \in 
\left [\frac{j}{n},  \frac{j+k}{n} \right ]$, it follows that
$|x_{n,k,j} - \frac{2j+k}{2n}| \leq \frac{k}{2n}$, and so
\begin{eqnarray*}
	\left | (V_{n,\infty}^{(k)} f - D_{n,\infty}^{(k)} f) (x) \right |
	&\leq &
	\sum_{j=0}^\infty \left |f(x_{n,k,j} )- f \left (\frac{2j+k}{2n} \right ) \right | p_{n+kc,j} (x)
	\\
	& \leq &
	\omega \left ( f; \frac{k}{2n} \right ) , \mbox{ for all } x \geq 0.
\end{eqnarray*}
This proves (\ref{eq3-15}).
\end{proof}
Let us return to $\Var{A_{n,\rho,j}^{(k)}}$, $E_{n,\rho}^{(k)}$ and $\Var_x{V_{n,\rho}^{(k)}}$.
Their roles were illustrated in Theorem \ref{thm3-1}; some properties of them are presented in the following result.
\begin{theorem}
\label{thm3-3}
$\Var{A_{n,\rho,j}^{(k)}}$, $E_{n,\rho}^{(k)}$ and $\Var_x{V_{n,\rho}^{(k)}}$ are:
\begin{itemize}
\item[(a)] increasing with respect to $k$,
\item[(b)] increasing with respect to $c$,
\item[(c)] decreasing with respect to $\rho$.
\end{itemize}
\end{theorem}
\begin{proof}
(a) and (b) follow easily from (\ref{eq3-8}), (\ref{eq3-9}) and (\ref{eq3-10}).

Now let us denote $a:=\frac{kc}{n}$ and $b:=\frac{kc+c}{n}$. Then (\ref{eq3-8}) can be rewritten as
\begin{eqnarray*}
	\Var{A_{n,\rho,j}^{(k)}}
	& = &
	\frac{1}{6n^3}
	\left \{ 3 [k(2jc+n)+2j(jc+n)] + k^2c \right \} \left ( \frac{\rho}{\rho-a} \right )^2 \frac{1}{\rho-b}
	\\
	& &
	+ \frac{k}{12n^2} \left \{ \frac{\rho^2+3}{(\rho-a)^2} \frac{\rho}{\rho-b} +
	\frac{2}{\rho-a} \frac{1}{\rho-b}  \right \} ,
\end{eqnarray*}
and this shows that $\Var{A_{n,\rho,j}^{(k)}}$ is decreasing w.r.t. $\rho$.

A similar transformation of (\ref{eq3-9}) (or an inspection of (\ref{eq2-7})) reveals that 
$E_{n,\rho}^{(k)}$ is also decreasing w.r.t. $\rho$.

Finally (\ref{eq3-10}) can be put under the form 
\begin{eqnarray*}
	\Var_x{V_{n,\rho}^{(k)}}
	& = &
	\frac{x(1+cx)(n+kc)}{n^2} \left ( \frac{\rho}{\rho-a} \right )^2 \frac{\rho+1}{\rho-b}
	+ \frac{kc+3n}{6n^3} \left ( \frac{\rho}{\rho-a} \right )^2 \frac{1}{\rho-b}
	\\
	& &
	+\frac{k}{12n^2} \left \{ \frac{\rho^2+3}{(\rho-a)^2}  \frac{\rho}{\rho-b}
	+\frac{2}{\rho-a} \frac{1}{\rho-b} \right \} ,
\end{eqnarray*}
which shows that $\Var_x{V_{n,\rho}^{(k)}}$ is decreasing w.r.t. $\rho$.
\end{proof}
\section{Some complementary results}
\label{sec4}
The starting point for this article was the following remark.

Let $n \in \mathbb{N}$, $c=-1$, $x \in [0,1]$, $1 \leq k < n$, $f \in C[0,1]$.
Then (\ref{eq3-17}) takes the form
\begin{equation}
\label{eq4-1}
	V_{n,\infty}^{(k)} f(x) = \sum_{j=0}^{n-k} f(x_{n,k,j}) p_{n-k,j} (x),
\end{equation}
for suitable $x_{n,k,j} \in \left [ \frac{j}{n},\frac{j+k}{n} \right ]$.
In (\ref{eq4-1}) $p_{n-k,j} (x)$ are the very classical Bernstein fundamental polynomials, i.~e.,
$$
	p_{n-k,j}(x) := {n-k \choose j } x^j (1-x)^{n-k-j}, \, x \in [0,1].
$$
The classical Bernstein polynomials are, of course,
\begin{equation}
\label{eq4-2}
	B_{n-k} f(x) :=\sum_{j=0}^{n-k} f \left ( \frac{j}{n-k} \right ) p_{n-k,j} (x).
\end{equation}
Both $x_{n,k,j}$ and $\frac{j}{n-k}$ are in $\left [ \frac{j}{n},\frac{j+k}{n} \right ]$, so
that $|x_{n,k,j}- \frac{j}{n-k}| \leq \frac{k}{n}$. Now from (\ref{eq4-1}) and (\ref{eq4-2}) we get
\begin{equation}
\label{eq4-3}
	\|V_{n,\infty}^{( k)} - B_{n-k} f\|_\infty \leq \omega \left ( f; \frac{k}{n} \right ) .
\end{equation}
In particular, $V_{n,\infty}^{( 1)} $ is the classical Kantorovich operator $K_{n-1}$ on $C[0,1]$.
So (\ref{eq4-3}) implies
$$
	\| K_{n-1} f - B_{n-1}f \|_\infty \leq \omega \left ( f; \frac{1}{n} \right ) , \, f \in C[0,1].
$$
Let us return to an arbitrary $c$ and to 
$$
		B_{n,\infty} f(x) :=\sum_{j=0}^{\infty} p_{n,j} (x) f \left ( \frac{j}{n} \right ).
$$
The sum $S_n(x) = S_{n,c} (x) := \sum_{j=0}^\infty p_{n,j}^2 (x) $ was investigated in a series of papers, see \cite{Rasa2015,Rasatbp} and the references given there.

It is known that it is logarithmically convex (see \cite{Rasatbp}) and
\begin{eqnarray*}
	S_{n,c} (x) 
	& = &
	\frac{1}{\pi} \int_0^1 \left ( t+(1-t)(1+2cx)^2 \right )^{-\frac{n}{c}} \frac{dt}{\sqrt{t(1-t)}}, \,
	c \not=0,
	\\
	S_{n,0} (x) 
	& = &
	\frac{1}{\pi} \int_{-1}^1 e^{-2nx(1+t)} \frac{dt}{\sqrt{1-t^2}}, \,
	\\
	S_{n,c} (x) 
	& \leq &
	(4(n+c)x(1+cx)+1)^{-\frac{n}{2(n+c)}} .
\end{eqnarray*}
Moreover (see \cite{Rasa2015,GonskaRasaRusu2014}),
$$
	|B_{n,\infty} (fg) (x) - B_{n,\infty} f(x)B_{n,\infty} g(x) |
	\leq \frac{1}{2} (1-S_{n,c}(x)) \osc_n{(f)} \osc_n{(g}),
$$
where $\osc_n{(f)} := \sup{\{|f(\frac{j}{n}) - f(\frac{i}{n}) |: i,j \in \mathbb{N}_0 \}}$.

Let $n$ and $x$ be fixed; consider the functional $ f \mapsto B_{n,\infty} f(x) $. Then, roughly speaking $1-S_{n,c} (x) $ shows how far is the functional from being a multiplicative functional, and ultimately how far is it from being a point evaluation.

Consider $(p_{n,0} (x), p_{n,1} (x),  \dots)$ as a probability distribution parameterized by $x$.
Then $1-S_{n,c}(x)$ is the associated Tsallis entropy.
Moreover, $S_n(x)$ is viewed as one of the indices measuring the inequality and diversity, i.~e., the degree of uniformness of the distribution (see \cite[pp. 556-559]{MarshallOlkinArnold2011}).

Margareta Heilmann\\
School of Mathematics and Natural Sciences
\\
University of Wuppertal
\\
Gau{\ss}stra{\ss}e 20
\\
D-42119 Wuppertal, Germany
\\
email: {heilmann@math.uni-wuppertal.de}
\\[5pt]
Fadel Nasaireh 
\\
Department of Mathematics
\\
Technical University
\\
Str. Memorandumului 28
\\
RO-400114 Cluj-Napoca,
Romania,
\\
email: {fadelnasierh@gmail.com}
\\[5pt]
Ioan Ra\c{s}a
\\
Department of Mathematics
\\
Technical University
\\
Str. Memorandumului 28
\\
RO-400114 Cluj-Napoca,
Romania
\\
email: {Ioan.Rasa@math.utcluj.ro}
\end{document}